\documentclass{amsart}%
\usepackage[latin9]{inputenc}
\usepackage{amstext}
\usepackage{amsthm}
\usepackage{amssymb}
\usepackage{amsfonts}
\usepackage[pdftex]{graphicx}
\usepackage{amsmath}
\usepackage{hyperref}%
\providecommand{\U}[1]{\protect\rule{.1in}{.1in}}
\newtheorem{theorem}{Theorem}
\theoremstyle{plain}

\newtheorem{corollary}{Corollary}

\newtheorem{lemma}{Lemma}

\newtheorem{problem}{Problem}

\newtheorem{remark}{Remark}

\numberwithin{equation}{section}
\begin{document}
\title{A survey of the Hornich-Hlawka inequality}
\author{Dan-\c{S}tefan Marinescu}
\address{National College "Iancu de Hunedoara", Hunedoara, Romania}
\email{marinescuds@gmail.com}
\author{Constantin P. Niculescu}
\address{Department of Mathematics, University of Craiova, Craiova 200585, Romania}
\email{constantin.p.niculescu@gmail.com}
\thanks{Corresponding author: Constantin P. Niculescu}
\date{June 30, 2024}

\begin{abstract}
In this survey, we review the many faces of the Hornich-Hlawka inequality.
Several open problems that seem of utmost interest are mentioned.

\end{abstract}
\maketitle

\section{Introduction}

The Hornich-Hlawka inequality, in its original form, asserts that%
\begin{equation}
\left\Vert \mathbf{x}\right\Vert _{2}+\left\Vert \mathbf{y}\right\Vert
_{2}+\left\Vert \mathbf{z}\right\Vert _{2}+\left\Vert \mathbf{x}%
+\mathbf{y}+\mathbf{z}\right\Vert _{2}\geq\left\Vert \mathbf{x}+\mathbf{y}%
\right\Vert _{2}+\left\Vert \mathbf{y}+\mathbf{z}\right\Vert _{2}+\left\Vert
\mathbf{z}+\mathbf{x}\right\Vert _{2} \tag{$HH$}\label{HH}%
\end{equation}
for every triplet $\mathbf{x},\mathbf{y},\mathbf{z}$ of vectors in the
Euclidean space $\mathbb{R}^{N}$ and thus in any Hilbert space, real or
complex. Here (and what follows), $\left\Vert \cdot\right\Vert _{2}$ denotes
the Hilbertian norm.

The story of this inequality started ca. 80 years ago, when H. Hornich
\cite{Ho1942} mentioned it\ as a consequence of his main result. He referred
to E. Hlawka for a direct algebraic argument saying that \emph{Für den
Spezialfall} $m=1$, $n=2$ \emph{hat Herr HLAWKA mir einen rein algebraischen
Beweis angegeben} (in translation: For the special case $m=1$, $n=2$, Mr.
Hlawka gave me one purely algebraic proof). Hlawka's argument is both short
and unexpected. It took into account the identity,%
\begin{equation}
\left\Vert \mathbf{x}\right\Vert _{2}^{2}+\left\Vert \mathbf{y}\right\Vert
_{2}^{2}+\left\Vert \mathbf{z}\right\Vert _{2}^{2}+\left\Vert \mathbf{x}%
+\mathbf{y}+\mathbf{z}\right\Vert _{2}^{2}=\left\Vert \mathbf{x}%
+\mathbf{y}\right\Vert _{2}^{2}+\left\Vert \mathbf{y}+\mathbf{z}\right\Vert
_{2}^{2}+\left\Vert \mathbf{z}+\mathbf{x}\right\Vert _{2}^{2}, \tag{$%
Fr$}\label{Fr}%
\end{equation}
(used by M. Fréchet \cite{Fr1935} to characterize the inner product spaces) to
derive the relation%
\begin{multline*}
\left(  \left\Vert \mathbf{x}\right\Vert _{2}+\left\Vert \mathbf{y}\right\Vert
_{2}+\left\Vert \mathbf{z}\right\Vert _{2}+\left\Vert \mathbf{x}%
+\mathbf{y}+\mathbf{z}\right\Vert _{2}-\left\Vert \mathbf{x}+\mathbf{y}%
\right\Vert _{2}-\left\Vert \mathbf{y}+\mathbf{z}\right\Vert _{2}-\left\Vert
\mathbf{z}+\mathbf{x}\right\Vert _{2}\right) \\
\times\left(  \left\Vert \mathbf{x}\right\Vert _{2}+\left\Vert \mathbf{y}%
\right\Vert _{2}+\left\Vert \mathbf{z}\right\Vert _{2}+\left\Vert
\mathbf{x}+\mathbf{y}+\mathbf{z}\right\Vert _{2}\right) \\
=(\left\Vert \mathbf{x}\right\Vert _{2}+\left\Vert \mathbf{y}\right\Vert
_{2}-\left\Vert \mathbf{x}+\mathbf{y}\right\Vert _{2})\left(  \left\Vert
\mathbf{z}\right\Vert _{2}+\left\Vert \mathbf{x}+\mathbf{y}+\mathbf{z}%
\right\Vert _{2}-\left\Vert \mathbf{x}+\mathbf{y}\right\Vert _{2}\right) \\
+(\left\Vert \mathbf{y}\right\Vert _{2}+\left\Vert \mathbf{z}\right\Vert
_{2}-\left\Vert \mathbf{y}+\mathbf{z}\right\Vert _{2})\left(  \left\Vert
\mathbf{x}\right\Vert _{2}+\left\Vert \mathbf{x}+\mathbf{y}+\mathbf{z}%
\right\Vert _{2}-\left\Vert \mathbf{y}+\mathbf{z}\right\Vert _{2}\right) \\
+(\left\Vert \mathbf{z}\right\Vert _{2}+\left\Vert \mathbf{x}\right\Vert
_{2}-\left\Vert \mathbf{z}+\mathbf{x}\right\Vert _{2})\left(  \left\Vert
\mathbf{y}\right\Vert _{2}+\left\Vert \mathbf{x}+\mathbf{y}+\mathbf{z}%
\right\Vert _{2}-\left\Vert \mathbf{z}+\mathbf{x}\right\Vert _{2}\right)  .
\end{multline*}
The proof ends by noticing that the right hand side is nonnegative due to the
triangle inequality.

The weighted form of the Hornich-Hlawka's inequality is a direct consequence
of the unweighted form $(HH):$
\begin{multline*}
\frac{\lambda\left\Vert \mathbf{x}\right\Vert _{2}+\mu\left\Vert
\mathbf{y}\right\Vert _{2}+\nu\left\Vert \mathbf{z}\right\Vert _{2}}%
{\lambda+\mu+\nu}+\left\Vert \frac{\lambda\mathbf{x}+\mu\mathbf{y}%
+\nu\mathbf{z}}{\lambda+\mu+\nu}\right\Vert _{2}\\
\geq\frac{1}{\lambda+\mu+\nu}\left[  \left\Vert \lambda\mathbf{x}%
+\mu\mathbf{y}\right\Vert _{2}+\left\Vert \mu\mathbf{y}+\nu\mathbf{z}%
\right\Vert _{2}+\left\Vert \nu\mathbf{z}+\lambda\mathbf{x}\right\Vert
_{2}\right] \\
=\frac{1}{\lambda+\mu+\nu}\left[  \left(  \lambda+\mu\right)  \left\Vert
\frac{\lambda\mathbf{x}+\mu\mathbf{y}}{\lambda+\mu}\right\Vert _{2}\right. \\
\left.  +\left(  \mu+\nu\right)  \left\Vert \frac{\mu\mathbf{y}+\nu\mathbf{z}%
}{\mu+\nu}\right\Vert _{2}+\left(  \nu+\lambda\right)  \left\Vert \frac
{\nu\mathbf{z}+\lambda\mathbf{x}}{\nu+\lambda}\right\Vert _{2}\right]  ,
\end{multline*}
for all $\mathbf{x},\mathbf{y},\mathbf{z\in}\mathbb{R}^{N}$ and all scalars
$\lambda,\mu,\nu>0.$

D. M. Smiley and M. F. Smiley \cite{SmSm} (as well as A. Simon and P. Volkmann
\cite{SV}) used the terminology of \emph{quadrilateral space} for what we call
here a Hornich-Hlawka space. A simple geometric argument for their terminology
is offered by the solution of R. Tauraso to the problem \cite{Giu}. He proved
that in the Euclidean plane this inequality is equivalent to the following
property of the quadrilaterals (convex or not): \emph{Let} $ABCD$ \emph{be a
quadrilateral}, $E$ \emph{the midpoint of} $AC$, \emph{and} $F$ \emph{the
midpoint of} $BD$. \emph{Then}%
\[
\left\Vert AB\right\Vert _{2}+\left\Vert BC\right\Vert _{2}+\left\Vert
CD\right\Vert _{2}+\left\Vert DA\right\Vert _{2}\geq\left\Vert AC\right\Vert
_{2}+\left\Vert BD\right\Vert _{2}+2\left\Vert EF\right\Vert _{2}.
\]
The proof extends verbatim to all Euclidean spaces $\mathbb{R}^{N}$ with
$N\geq2.$ Even more, according to Remark \ref{rem_2dim}, in the 2-dimensional
case, the Euclidean norm can be replaced by any other norm on $\mathbb{R}%
^{2}.$

For the full information of our readers, we have included some information
about the lives and achievements of Hornich and Hlawka in the Bibliography
section. See \cite{Hor} and \cite{Ti2009}.

Stigler's law of eponymy states that no scientific discovery is named after
its original discoverer and the case of this inequality is not an exception.
Indeed, some authors (\cite{Levi}, \cite{Res}) refer to (HH) as the
\emph{Hornich-Hlawka inequality}, while others (\cite{Mit1970}, \cite{MPF}
\cite{TTW}, \cite{Wada}) to the \emph{Hlawka inequality. }We adopt here the
first variant, which seems better motivated. No matter how we refer to the
inequality (HH), a fact is certain: this is a perennial subject which continue
to attract a great deal of attention and our paper is aimed to emphasize the
mathematics around it. To do this we survey the known results and point out
the wealth of rather important open problems that are out there, at least some
of which should be fairly elementary for someone with the right background.

\section{Hornich-Hlawka inequality by transfer from an $L^{1}$ space}

Looking at the aforementioned proof of the (HH) inequality, the reader will
inevitably react as G. Pólya noted in the Preface of his influential book How
to Solve It : \emph{Yes, the solution seems to work, it appears to be correct.
But how is it possible to invent such a solution? }See \cite{Pol}, p. vi.

Many norm inequalities in real analysis follow from appropriate inequalities
for real numbers by integration, after substituting the real variables with
integrable functions. Think for example at the standard proof of Hölder's
inequality, which appears as a consequence of Young's inequality. See
\cite{NP2018}, p. 13.

Surprisingly, the same recipe works in the case of Hornich-Hlawka inequality.
Indeed, by taking into account the inequality%
\begin{equation}
|x|+|y|+|z|+|x+y+z|\geq|x+y|+|y+z|+|z+x|\text{\quad for all }x,y,z\in
\mathbb{R}, \label{HHReal}%
\end{equation}
which is an easy consequence of the ordering properties of $\mathbb{R},$ one
can easily prove the Hornich-Hlawka inequality not only in the case of Hilbert
spaces but also for a large variety of spaces. Notice also that the inequality
\ref{HHReal} plays a central role in the original proof of Popoviciu's
inequality. See \cite{Pop1965}.

For convenience, we will call a (real or complex) Banach space
$E=(E,\left\Vert \cdot\right\Vert )$ a \emph{Hornich-Hlawka space} if its norm
verifies the inequality%
\begin{equation}
\left\Vert \mathbf{x}\right\Vert +\left\Vert \mathbf{y}\right\Vert +\left\Vert
\mathbf{z}\right\Vert +\left\Vert \mathbf{x}+\mathbf{y}+\mathbf{z}\right\Vert
\geq\left\Vert \mathbf{x}+\mathbf{y}\right\Vert +\left\Vert \mathbf{y}%
+\mathbf{z}\right\Vert +\left\Vert \mathbf{z}+\mathbf{x}\right\Vert
\label{HHineq}%
\end{equation}
for every triplet of elements $\mathbf{x},\mathbf{y},\mathbf{z}$ of $E.$

An easy but important remark is the membership of every $L^{1}$ space to this
class of Banach spaces. The next result refers to the more sophisticated case
of (Bochner integrable) vector-valued functions. See the book of Diestel and
Uhl \cite{DU}, Chapter II, Section 2, for a quick introduction to this theory
of integration.

\begin{lemma}
\label{lemL1} Let $E$ be a Hornich-Hlawka space and $(\Omega,\Sigma,\mu)$ a
measure space. Then the Banach space $L_{E}^{1}(\Omega,\Sigma,\mu),$ of all
Bochner integrable functions $f:\Omega\rightarrow E$ with respect to $\mu,$ is
also a Hornich-Hlawka space.
\end{lemma}

In particular, every space $L^{1}(\Omega,\Sigma,\mu)=L_{\mathbb{R}}^{1}%
(\Omega,\Sigma,\mu)$ is Hornich-Hlawka.

\begin{proof}
Indeed, for every\ triplet of functions $\mathbf{f},\mathbf{g},\mathbf{h}\in
L_{E}^{1}(\Omega,\Sigma,\mu),$%
\begin{multline*}
\left\Vert \mathbf{f}(\omega)\right\Vert +\left\Vert \mathbf{g}(\omega
)\right\Vert +\left\Vert \mathbf{h}(\omega)\right\Vert +\left\Vert
\mathbf{f}(\omega)+\mathbf{g}(\omega)+\mathbf{h}(\omega)\right\Vert \\
\geq\left\Vert \mathbf{f}(\omega)+\mathbf{g}(\omega)\right\Vert +\left\Vert
\mathbf{g}(\omega)+\mathbf{h}(\omega)\right\Vert +\left\Vert \mathbf{h}%
(\omega)+\mathbf{f}(\omega)\right\Vert \text{\quad}%
\end{multline*}
for all $\omega\in\Omega,$ whence by integration we obtain%
\begin{multline*}
\int_{\Omega}\left\Vert \mathbf{f}(\omega)\right\Vert \mathrm{d}\mu
+\int_{\Omega}\left\Vert \mathbf{g}(\omega)\right\Vert \mathrm{d}\mu
+\int_{\Omega}\left\Vert \mathbf{h}(\omega)\right\Vert \mathrm{d}\mu
+\int_{\Omega}\left\Vert \mathbf{f}(\omega)+\mathbf{g}(\omega)+\mathbf{h}%
(\omega)\right\Vert \mathrm{d}\mu\\
\geq\int_{\Omega}\left\Vert \mathbf{f}(\omega)+\mathbf{g}(\omega)\right\Vert
\mathrm{d}\mu+\int_{\Omega}\left\Vert \mathbf{g}(\omega)+\mathbf{h}%
(\omega)\right\Vert \mathrm{d}\mu+\int_{\Omega}\left\Vert \mathbf{h}%
(\omega)+\mathbf{f}(\omega)\right\Vert \mathrm{d}\mu.
\end{multline*}
that is,
\[
\left\Vert \mathbf{f}\right\Vert _{L_{E}^{1}}+\left\Vert \mathbf{g}\right\Vert
_{L_{E}^{1}}+\left\Vert \mathbf{h}\right\Vert _{L_{E}^{1}}+\left\Vert
\mathbf{f}+\mathbf{g}+\mathbf{h}\right\Vert _{L_{E}^{1}}\geq\left\Vert
\mathbf{f}+\mathbf{g}\right\Vert _{L_{E}^{1}}+\left\Vert \mathbf{g}%
+\mathbf{h}\right\Vert _{L_{E}^{1}}+\left\Vert \mathbf{h}+\mathbf{f}%
\right\Vert _{L_{E}^{1}}.
\]

\end{proof}

\begin{corollary}
\label{corL1}Under the hypotheses of Lemma \emph{\ref{lemL1}}, every Banach
space that embeds linearly and isometrically into a space $L_{E}^{1}%
(\Omega,\Sigma,\mu)$ is a Hornich-Hlawka space.
\end{corollary}

The original Hornich-Hlawka inequality can be derived from Corollary
\ref{corL1} by noticing that the Euclidean space $\mathbb{R}^{N}$ embeds
linearly and isometrically into the Lebesgue space $L^{1}(S^{N-1}),$
associated to the normalized Haar measure $m$ on the unit sphere $S^{N-1}$ of
$\mathbb{R}^{N},$ invariant to rotations. The embedding $\Phi:\mathbb{R}%
^{N}\rightarrow L^{1}(S^{N-1})$ is defined via the formula
\[
\Phi(\mathbf{u})=\langle\mathbf{u},\cdot\rangle.
\]
To check this, let $\mathbf{u}_{0}\in S^{N-1}$ arbitrarily fixed and put
\[
C=\left\Vert \Phi(\mathbf{u}_{0})\right\Vert _{L^{1}}=\int_{S^{N-1}}\left\vert
\langle\mathbf{u}_{0},\mathbf{x}\rangle\right\vert \mathrm{d}m.
\]
By rotational invariance, $C=\int_{S^{N-1}}\left\vert \langle\mathbf{u}%
,\mathbf{x}\rangle\right\vert \mathrm{d}m$ for all $\mathbf{u}\in S^{N-1}$,
whence%
\[
\left\Vert \Phi(\mathbf{u})\right\Vert _{L^{1}}=\left\Vert \mathbf{u}%
\right\Vert _{2}\text{\quad for all }\mathbf{u\in}\mathbb{R}^{N}.
\]

This way the validity of the Hornich-Hlawka inequality is obtained by
"transfer" from the case $L^{1}$ space. We will come back to the idea of using
the Haar measures for deriving Hornich-Hlawka type inequalities in Section 6.

As noticed H. S. Witsenhausen \cite{W1973}, the embeddability in an $L^{1}$
space is equivalent to the property%

\[
\sum_{i,j=1}^{n}w_{i}w_{j}\left\Vert \mathbf{x}_{i}-\mathbf{x}_{j}\right\Vert
\leq0,
\]
for $w_{i}$ integers of sum 1, any $n$ and any $n$ points $\mathbf{x}_{k}.$

In the next section we will show that all spaces $L^{p}\left(  [0,1]\right)  $
with $p\in(1,2)$ can be embedded into an $L^{1}$ space.

\section{The Levy transfer theorem}

Recall that any $L^{p}$ space is associated to a measure space $(\Omega
,\Sigma,\mu)$\ and consists of all (equivalence classes) of measurable
functions on $\Omega$ such that%
\[
\left\Vert f\right\Vert _{p}=\left(  \int_{\Omega}\left\vert f\right\vert
^{p}d\mu\right)  ^{1/p}<\infty\text{\quad for }0<p<\infty
\]
and%
\[
\left\Vert f\right\Vert _{\infty}=\inf\left\{  C>0:\left\vert f(\omega
)\right\vert \leq C\text{ for almost every }\omega\right\}  <\infty\text{\quad
for }p=\infty.
\]

A characterization of the finite-dimensional Banach spaces that can be
embedded in an $L^{P}$ space with $p\in(0,\infty),$ was found by P. Levy (see
\cite{Levy}, §\ 63) in connection with his work on the theory of stable processes.

\begin{theorem}
\label{thmLevy}An $N$-dimensional real Banach space $E=(\mathbb{R}%
^{N},\left\Vert \cdot\right\Vert )$ embeds isometrically in $L^{p}(0,1)$
$(p>0)$ if and only if there exists a finite Borel measure $\mu$ on the unit
sphere $S^{N-1}$ in $\mathbb{R}^{N}$ so that, for every $x\in E,$%
\[
\left\Vert x\right\Vert ^{p}=\int_{\Omega}\left\vert \langle x,\xi
\rangle\right\vert ^{p}d\mu(\xi).
\]

\end{theorem}

The last formula is called by A. Koldobsky \cite{Kol1999} the
\emph{Blaschke-Levy representation}. See \cite{Kol1992}, Theorem 2, for a
generalization. For $p\in(0,2],$ this representation exists if and only if the
function $\exp(-\left\Vert x\right\Vert ^{p})$ is positive definite and,
hence, it is the characteristic function of a symmetric stable measure in
$\mathbb{R}^{N}.$

The work of P. Levy \cite{Levy}, I. J. Schoenberg \cite{Sch}, M. Kadets
\cite{Kad}, C. Herz \cite{Herz1971} and J. Bretagnolle, D. Dacunha-Castelle
and J. L. Krivine \cite{BDK} have prove the existence of the isometric
embeddings of the real spaces $\ell^{q}(N),~\ell^{q}$ and $L^{q}(0,1)$ in
$L^{p}(0,1)$ for $0<p\leq q\leq2.$ Here $\ell^{q}(N)$ is the the space
$\mathbb{R}^{N}$ endowed with the norm%
\[
\left\Vert \mathbf{x}\right\Vert _{q}=(\left\vert x_{1}\right\vert ^{q}%
+\cdots+\left\vert x_{N}\right\vert ^{q})^{1/q}%
\]
and $\ell^{q}$ denotes its sequence space analogue.

Every separable real Hilbert space is isometric to a subspace of $L^{p}(0,1),$
whenever $p>0$.$~$See J. Lindenstrauss and A. Pe\l czy\'{n}ski \cite{LP}%
,$~$Corollary 1, p. 312.

\begin{remark}
\label{rem_2dim}Every $2$-dimensional real Banach space is isometric to a
subspace of $L^{1}(0,1).$ This was proved independently by C. Herz
\emph{\cite{Herz}}, Th. S. Ferguson \emph{\cite{Ferg}} and J. Lindenstrauss
\emph{\cite{Lind}}. As a consequence, the Hornich-Hlawka inequality
\emph{(HH)} works for all $2$-dimensional real Banach spaces.
\end{remark}

\begin{remark}
L. Dor has shown that $\ell^{p}(3)$ is not isometric to any subspace of
$L^{1}(0,1)$ if $p>2.$ See \emph{\cite{Dor}}, Corollary of Theorem $1.5$, p.
$265$. Combined with the discussion above, it follows that none of the spaces
$L^{p}(0,1)$ and $\ell^{p}$ with $p>2$ can be isometrically embedded into a
space $L^{1}(0,1).$
\end{remark}

The Hornich-Hlawka inequality has a finite dimensional character. It concerns
\emph{only} the 3-dimensional subspaces of a Banach space. As was noticed by
A. Neyman \cite{Ney} (Corollary 3.4) there is no finite list of inequalities
that characterizes embeddability, neither in $L^{1}(0,1)$ nor in any other
$L^{p}$ space unless $p=2$. This is accomplished by proving that for every
$1\leq p<\infty$, $p\neq2$, and every $k\geq2$ there is a $(k+1)$-dimensional
normed space that is not embeddable in $L^{p}$ and all of its $k$-dimensional
subspaces are embeddable in $L^{p}.$ A consequence is the following result:

\begin{theorem}
There exist finite dimensional Hornich-Hlawka spaces that cannot be embedded
isometrically in $L^{1}(0,1).$
\end{theorem}

It is straightforward that none of the spaces $\ell^{p}(3)$ $(\mathbb{R}^{3}$
with the norm of index $p)$ for $p\in(\left(  \log3\right)  /\left(
\log3-\log2\right)  ,\infty]$ is Hornich-Hlawka. For, check the inequality for
$\mathbf{x}=(-1,1,1)$, $\mathbf{y}=(1,-1,1)$ and $\mathbf{z}=(1,1,-1).$ Notice
that $\left(  \log3\right)  /\left(  \log3-\log2\right)  =\allowbreak
2.\,\allowbreak709\,511\,291...$ .

As a consequence, all spaces $L^{p}(0,1)$ with $p$ in the same range of values
fail to be Hornich-Hlawka. Even more, it seems very likely that the answer to
the following problem is positive.

\begin{problem}
\label{probl1}Prove that actually none of the spaces $\ell^{p}(3)$ with
$p\in(2,\infty]$ is Hornich-Hlawka.
\end{problem}

A positive answer to Problem \ref{probl1} would implies that none of the
spaces $L^{p}(0,1)$ with $p\in(2,\infty]$ is Hornich-Hlawka.

How apart can be the class of Hornich-Hlawka spaces from the rest of Banach spaces?

\begin{problem}
\label{probNonHH}Are there Banach spaces whose all finite dimensional
subspaces of dimension greater than or equal to $3$ are not Hornich-Hlawka?
\end{problem}

A positive solution to Problem \ref{probNonHH}, if it exists, will probably be
an amalgamation of the spaces $\ell^{3}(3),\ell^{4}(3),\ell^{5}(3),\ldots$ but
we lack a clear argument.

A natural generalization of $L^{p}$ spaces is provided by the class of
$L^{p(\cdot)}$ spaces \emph{with variable exponent}. They differ from the
classical $L^{p}$ spaces in that the exponent $p$ is not constant but a
function from the measure space under attention to $[1,\infty]$. See the
monograph of L. Diening, P. Harjulehto, P. Hästö and M. Ruzicka
\cite{DHHR2011} for a thorough presentation.

\begin{problem}
\label{probl2}Is every $L^{p(\cdot)}(0,1)$ space with a variable exponent
$p(\cdot):(0,1)\rightarrow\lbrack1,2]$ a Hornich-Hlawka space?
\end{problem}

Let $p\in\lbrack1,\infty)$ and $k\geq1.$ The Sobolev space $W^{k,p}(0,1)$
consists of all functions $u\in L^{p}\left(  0,1\right)  $ which admit weak
derivatives $\partial_{\alpha}u\in L^{p}\left(  0,1\right)  $ for all multi
indices $\alpha$ with $\left\vert \alpha\right\vert \,\leq k.$ This is a
Banach space with respect to the norm%
\[
\left\Vert u\right\Vert _{W^{k,p}}=\left(  \sum\nolimits_{\left\vert
\alpha\right\vert \,\leq k}\int_{0}^{1}\left\vert \partial_{\alpha}\right\vert
^{p}\mathrm{d}x\right)  ^{1/p}.
\]
A nice overview of the theory of Sobolev spaces is offered by the book of M.
Willem \cite{Wil}.

\begin{theorem}
Let $k$ be a positive integer and $p\in\lbrack1,2].$ Then the spaces
$W^{k,p}(0,1)$ and $W_{0}^{k,p}(0,1)$ $($the closure in $W^{k,p}(0,1)$ of the
space of infinitely differentiable functions compactly supported in $\Omega),$
both are Hornich-Hlawka.
\end{theorem}

\begin{proof}
Since $L^{p}(0,1)$ embeds linearly and isometrically into $L^{1}(0,1),$ it
follows that a similar result works for the pair of spaces $L_{\mathbb{R}^{m}%
}^{p}(0,1)$ and $L_{\mathbb{R}^{m}}^{1}(0,1).$ According to Lemma \ref{lemL1},
this shows that $L_{\mathbb{R}^{m}}^{p}(0,1)$ is Hornich-Hlawka. The map%
\[
\Phi:W^{k,p}(0,1)\rightarrow L^{p}((0,1);\mathbb{R}^{m}),\text{ }%
\Phi(u)=\left(  \partial^{\alpha}u\right)  _{\left\vert \alpha\right\vert \leq
k}%
\]
is a linear isometry since $\left\Vert \Phi(u)\right\Vert _{p}=\left\Vert
u\right\Vert _{k,p}.$ Therefore $W^{k,p}(0,1)$ and its closed subspace
$W_{0}^{k,p}(0,1)$ both are Hornich-Hlawka.
\end{proof}

The aforementioned monograph of Diening, Harjulehto, Hästö and Ruzicka
\cite{DHHR2011} also includes a comprehensive presentation of the theory of
Sobolev spaces $W^{k,p(\cdot)}(0,1)$ of a variable exponent. Related to
Problem \ref{probl2} is the following question.

\begin{problem}
\label{probl3}Which Sobolev spaces $W^{k,p(\cdot)}(0,1)$ with a variable
exponent $p(\cdot):(0,1)\rightarrow\lbrack1,2]$ verify the Hornich-Hlawka inequality?
\end{problem}

\section{The Hornich-Hlawka inequality through finite representability}

The embeddability condition in Theorem \ref{thmLevy} can be relaxed using the
concepts of finite representability and of isomorphic embedding.

A Banach space $X$ is said to be \emph{finitely representable} in a Banach
space $Y$ if for every $\varepsilon>0$ and every subspace $M$ of $X$ there
exists a finite dimensional subspace $N$ of $Y$ and a linear isomorphism
$T:M\rightarrow N$ such that%
\[
(1-\varepsilon)\left\Vert x\right\Vert \leq\left\Vert T(x)\right\Vert
\leq(1+\varepsilon)\left\Vert x\right\Vert \text{\quad for all }x\in M.
\]

According to a famous result due to Dvoretzky, $\ell^{2}$ is finitely
representable in any Banach space. See Figiel \cite{Fi} for a short proof.

On the other hand, according to the principle of local reflexivity (see
\cite{LR1969}), the second dual of any Banach space $X$ is finitely
representable in $X$.

The following transfer theorem originates from Proposition 7.6 in \cite{LP}:

\begin{theorem}
\label{thmLPgen} Suppose that $A=(a_{ik})_{i,k}$ is a real $m\times r$ matrix,
$B=(b_{jk})_{j,k}$ is a real $n\times r$ matrix and $(\lambda_{i})_{i=1}^{m}$
and $(\mu_{j})_{j=1}^{n}$ are two families of positive numbers. If for some
$p>0,$%
\[
\sum\nolimits_{i=1}^{m}\lambda_{i}\left\vert \sum\nolimits_{k=1}^{r}%
a_{ik}t_{k}\right\vert ^{p}\leq\sum\nolimits_{j=1}^{n}\mu_{j}\left\vert
\sum\nolimits_{k=1}^{r}b_{jk}t_{k}\right\vert ^{p}%
\]
for all $t_{1},...,t_{r}\in\mathbb{R},$ then%
\[
\sum\nolimits_{i=1}^{m}\lambda_{i}\left\Vert \sum\nolimits_{k=1}^{r}%
a_{ik}\mathbf{x}_{k}\right\Vert _{X}^{p}\leq\sum\nolimits_{j=1}^{n}\mu
_{j}\left\Vert \sum\nolimits_{k=1}^{r}b_{jk}\mathbf{x}_{k}\right\Vert _{X}^{p}%
\]
for every family $\mathbf{x}_{1},...,\mathbf{x}_{r}$ of vectors in a Banach
space $X=(X,\left\Vert \cdot\right\Vert _{X})$ which is finitely representable
in $L^{p}(0,1).$
\end{theorem}

\begin{proof}
Indeed, we may reduce ourselves to the case where $X$ is finite dimensional
and for every $\varepsilon>0$ there exists a linear map $T:X\rightarrow
L^{p}(0,1)$ such that
\[
(1-\varepsilon)\left\Vert x\right\Vert ^{p}\leq\left\Vert T(x)\right\Vert
^{p}\leq(1+\varepsilon)\left\Vert x\right\Vert ^{p}%
\]
\ for all $x\in X.$ According to our hypothesis,
\[
\sum\nolimits_{i=1}^{m}\lambda_{i}\left\vert \langle\sum\nolimits_{k=1}%
^{r}a_{ik}\mathbf{x}_{k},\mathbf{u}\rangle\right\vert ^{p}\leq\sum
\nolimits_{j=1}^{n}\mu_{j}\left\vert \langle\sum\nolimits_{k=1}^{r}%
b_{jk}\mathbf{x}_{k},\mathbf{u}\rangle\right\vert ^{p},
\]
for every $\mathbf{u}\in S^{N-1},$ and an appeal to Theorem \ref{thmLevy}
easily yields the existence of a finite Borel measure $\mu$ on the unit sphere
$S^{N-1}$ in $\mathbb{R}^{N}$ such that \
\begin{multline*}
(1-\varepsilon)\sum\nolimits_{i=1}^{m}\lambda_{i}\left\Vert \sum
\nolimits_{k=1}^{r}a_{ik}\mathbf{x}_{k}\right\Vert _{X}^{p}\leq\sum
\nolimits_{i=1}^{m}\lambda_{i}\left\Vert \sum\nolimits_{k=1}^{r}%
a_{ik}\mathbf{x}_{k}\right\Vert _{p}^{p}\\
=\sum\nolimits_{i=1}^{m}\lambda_{i}\int_{S^{N-1}}\left\vert \langle
\sum\nolimits_{j=1}^{r}a_{ij}\mathbf{x}_{j},\mathbf{u}\rangle\right\vert
^{p}d\mu(\mathbf{u})\\
\leq\sum\nolimits_{j=1}^{n}\mu_{j}\int_{S^{N-1}}\left\vert \langle
\sum\nolimits_{k=1}^{r}b_{jk}\mathbf{x}_{k},\mathbf{u}\rangle\right\vert
^{p}d\mu(\mathbf{u})\\
=\sum\nolimits_{j=1}^{n}\mu_{j}\left\Vert \sum\nolimits_{k=1}^{r}%
b_{jk}\mathbf{x}_{k}\right\Vert _{p}^{p}\leq(1+\varepsilon)\sum\nolimits_{j=1}%
^{n}\mu_{j}\left\Vert \sum\nolimits_{k=1}^{r}b_{jk}\mathbf{x}_{k}\right\Vert
_{X}^{p}.
\end{multline*}
The proof ends by noticing that $\varepsilon>0$ was arbitrarily fixed.
\end{proof}

Theorem \ref{thmLPgen} implies the following new generalization of Hlawka's
inequality:%
\[
\left\Vert \mathbf{x}\right\Vert _{X}^{p}+\left\Vert \mathbf{y}\right\Vert
_{X}^{p}+\left\Vert \mathbf{z}\right\Vert _{X}^{p}+\left\Vert \mathbf{x}%
+\mathbf{y}+\mathbf{z}\right\Vert _{X}^{p}\geq\left\Vert \mathbf{x}%
+\mathbf{y}\right\Vert _{X}^{p}+\left\Vert \mathbf{y}+\mathbf{z}\right\Vert
_{X}^{p}+\left\Vert \mathbf{z}+\mathbf{x}\right\Vert _{X}^{p}%
\]
for all vectors in a Banach space $X=(X,\left\Vert \cdot\right\Vert _{X})$
which is finitely representable in a space $L^{p}(0,1)$ with $0<p\leq2.$

\section{The Hornich-Hlawka inequality in uniformly non-square spaces}

In connection with the famous work \cite{JN} of Jordan and von Neumann
concerning the inner product spaces, J. A. Clarkson \cite{Cl1937} has
introduced the von Neumann-Jordan constant $C_{NJ}(X)$ of a Banach space $X$
as
\[
C_{NJ}(X)=\sup\left\{  \frac{\left\Vert \mathbf{x}+\mathbf{y}\right\Vert
^{2}+\left\Vert \mathbf{x}-\mathbf{y}\right\Vert ^{2}}{2\left\Vert
\mathbf{x}\right\Vert ^{2}+2\left\Vert \mathbf{y}\right\Vert ^{2}}%
:\mathbf{x},\mathbf{y}\in X\text{ and }\left\Vert \mathbf{x}\right\Vert
+\left\Vert \mathbf{y}\right\Vert \neq0\right\}  .
\]

We have $1\leq C_{NJ}(X)\leq2$ for all Banach spaces $X$ and $C_{NJ}(X)=1$
\emph{if and only if} $X$ \emph{is a Hilbert space}. Notice also that
$C_{NJ}(X)<2$ for any uniformly convex space and $C_{NJ}(X)=2$ in the case of
spaces $L^{p}(\mathbb{R})$ with $p=1$\ or $p=\infty$ and the same is true in
the case of Banach spaces of continuous functions endowed with the sup norm.

For details see \cite{KMT} and \cite{KT}.

\begin{theorem}
In any Banach space,%
\[
C_{NJ}(X)\left(  \left\Vert x+y+z\right\Vert ^{2}+\left\Vert x\right\Vert
^{2}+\left\Vert y\right\Vert ^{2}+\left\Vert z\right\Vert ^{2}\right)
\geq\left\Vert x+y\right\Vert ^{2}+\left\Vert y+z\right\Vert ^{2}+\left\Vert
z+x\right\Vert ^{2}.
\]

Notice that in the case of Hilbert spaces this inequality reduces to the
Hornich-Hlawka inequality.
\end{theorem}

\begin{proof}
Indeed,%
\[
C_{NJ}(X)\left(  \left\Vert \mathbf{x}+y+z\right\Vert ^{2}+\left\Vert
x\right\Vert ^{2}\right)  \geq\frac{1}{2}\left\Vert 2x+y+z\right\Vert
^{2}+\frac{1}{2}\left\Vert y+z\right\Vert ^{2}%
\]
and%
\[
C_{NJ}(X)\left(  \left\Vert y\right\Vert ^{2}+\left\Vert z\right\Vert
^{2}\right)  \geq\frac{1}{2}\left\Vert y+z\right\Vert ^{2}+\frac{1}%
{2}\left\Vert y-z\right\Vert ^{2}.
\]

On the other hand,%
\begin{align*}
\frac{1}{2}\left\Vert 2x+y+z\right\Vert ^{2}+\frac{1}{2}\left\Vert
y-z\right\Vert ^{2}  &  =\frac{1}{2}\left\Vert \left(  x+y\right)  +\left(
x+z\right)  \right\Vert ^{2}+\frac{1}{2}\left\Vert \left(  x+y\right)
-(x+z)\right\Vert ^{2}\\
&  \geq\left\Vert x+y\right\Vert ^{2}+\left\Vert x+z\right\Vert ^{2}.
\end{align*}
Therefore%
\[
C_{NJ}(X)\left(  \left\Vert x+y+z\right\Vert ^{2}+\left\Vert x\right\Vert
^{2}+\left\Vert y\right\Vert ^{2}+\left\Vert z\right\Vert ^{2}\right)
\geq\left\Vert x+y\right\Vert ^{2}+\left\Vert y+z\right\Vert ^{2}+\left\Vert
z+x\right\Vert ^{2}.
\]

\end{proof}

A Banach space $X$ is called \emph{uniformly non-square} $($James
\emph{\cite{James1964}}$)$ if there exists a $\delta$ $\in(0,1)$ such that for
any $x$, $y$ in the unit sphere of $X$ either $\left\Vert \mathbf{x}%
+\mathbf{y}\right\Vert /2\leq1-\delta$ or $\left\Vert \mathbf{x}%
-\mathbf{y}\right\Vert /2\leq1-\delta$.

The constant%
\[
J(X)=\sup\{\min(\left\Vert \mathbf{x}+\mathbf{y}\right\Vert ,\left\Vert
\mathbf{x}-\mathbf{y}\right\Vert ):\left\Vert \mathbf{x}\right\Vert
=\left\Vert \mathbf{y}\right\Vert =1\}
\]
is called the \emph{non-square} or \emph{James constant} of $X$.

If $\dim X\geq2,$ then $\sqrt{2}\leq J(X)\leq2$ and $J(X)=\sqrt{2}$ if $X$ is
a Hilbert space. A Banach space $X$ is uniformly non-square if and only if
$J(X)<2.$ According to M. Kato, L. Maligranda, and Y. Takahashi
\emph{\cite{KMT}}, Theorem $3$,
\[
\frac{1}{2}J(X)^{2}\leq C_{NJ}(X)\leq\frac{J(X)^{2}}{\left(  J(X)-1\right)
^{2}+1},
\]
so the condition $C_{NJ}(X)<2$ is equivalent to the fact that $X$ is uniformly
non-square; see \emph{\cite{KMT}}, Proposition 1. As was noticed by James
\emph{\cite{James1964}}, any uniformly non-square space is reflexive.

The following Hornich-Hlawka type inequality that works in any Banach space
can be found in the paper by Y. Takahashi, S. E. Takahashi, and S. Wada
\cite{TTW2002}:

\begin{theorem}
Let $E$ be a Banach space and $r,s\geq1$ be two real numbers. Then
\begin{multline*}
\left(  \left\Vert \mathbf{x}+\mathbf{y}\right\Vert ^{s}+\left\Vert
\mathbf{y}+\mathbf{z}\right\Vert ^{s}+\left\Vert \mathbf{z}+\mathbf{x}%
\right\Vert ^{s}\right)  ^{1/s}\\
\leq2^{1-\left(  2/r\right)  }3^{1/s}\left(  \left\Vert \mathbf{x}\right\Vert
^{r}+\left\Vert \mathbf{y}\right\Vert ^{r}+\left\Vert \mathbf{z}\right\Vert
^{r}+\left\Vert \mathbf{x}+\mathbf{y}+\mathbf{z}\right\Vert ^{r}\right)
^{1/r}%
\end{multline*}
for all $\mathbf{x,y},\mathbf{z}\in E.$
\end{theorem}

\section{Levi's reduction theorem}

F. W. Levi \cite{Levi} found in 1949 a class of numerical inequalities that
generates vector inequalities simply by replacing the real variables with
vectors in an arbitrary Euclidean space and the absolute value with the
Euclidean norm. His result is based on the following lemma.

\begin{lemma}
\label{lem1}$(a)$ Let $\Phi:\mathbb{R}^{N}\times S^{N-1}\rightarrow\mathbb{R}$
be a continuous function whose restriction to $S^{N-1}\times S^{N-1}$ is
invariant under the action of the special orthogonal group \textrm{SO}$(N),$
that is,%
\[
\Phi\left(  S\mathbf{y},S\mathbf{x}\right)  =\Phi(\mathbf{y},\mathbf{x})
\]
for all $\mathbf{y},\mathbf{x}\in S^{N-1}$ and all orthogonal matrices $S$ of
determinant $1$ and dimension $N\times N.$ Then the integral%
\[
I(\mathbf{y})=\int_{S^{N-1}}\Phi(\mathbf{y},\mathbf{x})d\sigma(\mathbf{x}%
),\text{\quad}\mathbf{y}\in\mathbb{R}^{N},
\]
is constant on $S^{N-1}$.

$(b)$ Suppose in addition that $\Phi$ also verifies the following condition of
homogeneity:%
\[
\Phi(\lambda\mathbf{y},\mathbf{x})=\varphi(\lambda)\Phi(\mathbf{y}%
,\mathbf{x})
\]
for a certain continuous function $\varphi:[0,\infty)\rightarrow\mathbb{R}$
and all $\lambda\in\mathbb{R}_{+},$ $\mathbf{y}\in\mathbb{R}^{N}$ and
$\mathbf{x}\in S^{N-1}.$ Then%
\[
I(y)=\int_{S^{N-1}}\Phi(\mathbf{y},\mathbf{x})d\sigma(\mathbf{x})=C(\Phi
)\cdot\varphi(\left\Vert \mathbf{y}\right\Vert _{2})\quad\text{for all }%
y\in\mathbb{R}^{N},
\]
where $C(\Phi)$ is the value of constancy of $I(\mathbf{y})$ on $S^{N-1}$.
\end{lemma}

\begin{proof}
The constancy of $I(\mathbf{y})$ on $S^{N-1}$ is a consequence of the fact
that \textrm{SO}$(N)$ acts transitively on $S^{N-1}$ and the measure $d\sigma$
is invariant under the action of this group. The assertion $(b)$ is straightforward.
\end{proof}

Lemma \ref{lem1} yields the following transfer of real inequalities to their
vector companions in Euclidean spaces:

\begin{theorem}
\label{thm1}Suppose that $f:\mathbb{R}\rightarrow\mathbb{R}$ is a nonnegative
continuous function not identically zero and $\varphi:[0,\infty)\rightarrow
\mathbb{R}$ is a continuous function such that%
\[
f(\gamma t)=\varphi(\gamma)f(t)
\]
for all $\gamma\geq0$ and $t\in\mathbb{R}.$ If $A=(a_{ik})_{i,k}$ is a real
$m\times r$ matrix, $B=(b_{jk})_{j,k}$ is a real $n\times r$ matrix and
$(\lambda_{i})_{i=1}^{m}$ and $(\mu_{j})_{j=1}^{n}$ are two families of real
numbers such that%
\[
\sum\nolimits_{i=1}^{m}\lambda_{i}f\left(  \sum\nolimits_{k=1}^{r}a_{ik}%
t_{k}\right)  \leq\sum\nolimits_{j=1}^{n}\mu_{j}f\left(  \sum\nolimits_{k=1}%
^{r}b_{jk}t_{k}\right)
\]
for all $t_{1},...,t_{r}\in\mathbb{R},$ then%
\[
\sum\nolimits_{i=1}^{m}\lambda_{i}\varphi\left(  \left\Vert \sum
\nolimits_{k=1}^{r}a_{ik}\mathbf{x}_{k}\right\Vert _{2}\right)  \leq
\sum\nolimits_{j=1}^{n}\mu_{j}\varphi\left(  \left\Vert \sum\nolimits_{k=1}%
^{r}b_{jk}\mathbf{x}_{k}\right\Vert _{2}\right)
\]
for every family $\mathbf{x}_{1},...,\mathbf{x}_{r}$ of vectors in the
Euclidean space $\mathbb{R}^{N}$.
\end{theorem}

\begin{proof}
We apply Lemma \ref{lem1} to the function%
\[
\Phi(\mathbf{y},\mathbf{x})=f\left(  \langle\mathbf{y},\mathbf{x}%
\rangle\right)  ,
\]
by noticing that in this case the value of constancy $C(\Phi)$ is strictly positive.

Fix arbitrarily $\mathbf{u}\in S^{r-1}$ and put $t_{k}=\langle\mathbf{x}%
_{k},\mathbf{u}\rangle.$ According to our hypotheses,
\[
\sum\nolimits_{i=1}^{m}\lambda_{i}f\left(  \langle\sum\nolimits_{k=1}%
^{r}a_{ik}\mathbf{x}_{k},\mathbf{u}\rangle\right)  \leq\sum\nolimits_{j=1}%
^{n}\mu_{j}f\left(  \langle\sum\nolimits_{k=1}^{r}b_{jk}\mathbf{x}%
_{k},\mathbf{u}\rangle\right)  ,
\]
whence, according to Lemma \ref{lem1} we infer that%
\begin{multline*}
C(\Phi)\sum\nolimits_{i=1}^{m}\lambda_{i}\varphi\left(  \left\Vert
\sum\nolimits_{k=1}^{r}a_{ik}\mathbf{x}_{k}\right\Vert _{2}\right)
=\sum\nolimits_{i=1}^{m}\lambda_{i}\int_{S^{N-1}}f\left(  \langle
\sum\nolimits_{j=1}^{r}a_{ij}\mathbf{x}_{j},\mathbf{u}\rangle\right)
d\mathbf{u}\\
\leq\sum\nolimits_{j=1}^{n}\mu_{j}\int_{S^{N-1}}f\left(  \langle
\sum\nolimits_{k=1}^{r}b_{jk}\mathbf{x}_{k},\mathbf{u}\rangle\right)
d\mathbf{u}=C(\Phi)\sum\nolimits_{j=1}^{n}\mu_{j}\varphi\left(  \left\Vert
\sum\nolimits_{k=1}^{r}b_{jk}\mathbf{x}_{k}\right\Vert _{2}\right)  .
\end{multline*}
The proof ends by simplifying the end sides by $C(\Phi).$
\end{proof}

In the particular case when $f(t)=\varphi(t)=\left\vert t\right\vert ^{p}$ and
$p>0$, Theorem \ref{thm1} reads as follows:

\begin{corollary}
\label{cor1}Suppose that $A=(a_{ik})_{i,k}$ is a real $m\times r$ matrix,
$B=(b_{jk})_{j,k}$ is a real $n\times r$ matrix and $(\lambda_{i})_{i=1}^{m}$
and $(\mu_{j})_{j=1}^{n}$ are two families of real numbers. If for some
$\alpha>0,$%
\[
\sum\nolimits_{i=1}^{m}\lambda_{i}\left\vert \sum\nolimits_{k=1}^{r}%
a_{ik}t_{k}\right\vert ^{\alpha}\leq\sum\nolimits_{j=1}^{n}\mu_{j}\left\vert
\sum\nolimits_{k=1}^{r}b_{jk}t_{k}\right\vert ^{\alpha}%
\]
for all $t_{1},...,t_{r}\in\mathbb{R},$ then%
\[
\sum\nolimits_{i=1}^{m}\lambda_{i}\left\Vert \sum\nolimits_{k=1}^{r}%
a_{ik}\mathbf{x}_{k}\right\Vert _{2}^{\alpha}\leq\sum\nolimits_{j=1}^{n}%
\mu_{j}\left\Vert \sum\nolimits_{k=1}^{r}b_{jk}\mathbf{x}_{k}\right\Vert
_{2}^{\alpha}%
\]
for every family $\mathbf{x}_{1},...,\mathbf{x}_{r}$ of vectors in the
Euclidean space $\mathbb{R}^{N}$.
\end{corollary}

\begin{corollary}
$($Levi's reduction theorem \emph{\cite{Levi}}$)$\label{corLevi} Every
piecewise linear inequality of the form
\[
\sum_{i=1}^{n}\lambda_{i}\left\vert \sum_{k=1}^{r}a_{ik}t_{k}\right\vert
\geq0,
\]
which holds for all $n$-tuples $t_{1},...,t_{n}$ of real numbers remains true
when absolute value function is replaced by the Euclidean norm and
$t_{1},...,t_{r}$ are arbitrary elements of $\mathbb{R}^{N}$.
\end{corollary}

It is straightforward that both Corollaries \ref{cor1} and \ref{corLevi}
extend verbatim to the case of finite families of vectors in an arbitrary
inner product spaces.

Levi deduced the Hornich-Hlawka inequality as a consequence of Corollary
\ref{corLevi}, but his approach also yields the following more general
inequality,%
\begin{equation}
\sum\left\Vert \pm x_{1}\pm x_{2}\pm\cdots\pm x_{n}\right\Vert _{2}\geq
2\binom{n-1}{[(n-1)/2]}\sum_{k=1}^{n}\left\vert |x_{k}\right\Vert _{2},
\label{Hineq}%
\end{equation}
where the summation on the left is taken over all $2^{n}$ permutation of the
$\pm$ signs and $[t]$ is the integer part function (that is, the largest
integer less than or equal to $t$).

The R\u{a}dulescu brothers applied Levi's reduction theorem to prove the
following generalization of Dobrushin's inequality (see \cite{RR}, Theorem
3.3): Let $A=(a_{ij})_{i,j}$ be a real $m\times n$ dimensional matrix and $H$
a real Hilbert space. Then%
\begin{equation}
\sum_{i=1}^{m}\left\Vert \sum_{j=1}^{n}a_{ij}x_{j}\right\Vert \leq\bar{\alpha
}_{m,n}(A)\sum_{k=1}^{n}\left\Vert \mathbf{x}_{k}\right\Vert +(\left\Vert
A\right\Vert _{1}-\bar{\alpha}_{m,n}(A))\left\Vert \sum_{k=1}^{n}%
\mathbf{x}_{k}\right\Vert \label{HHD}%
\end{equation}
for all $\mathbf{x}_{1},...,\mathbf{x}_{n}\in H.$ Here%
\[
\bar{\alpha}_{m,n}(A)=\frac{1}{2}\max_{r,s}\sum_{i=1}^{m}\left\vert
a_{ir}-a_{is}\right\vert \quad\text{and\quad}\left\Vert A\right\Vert _{1}%
=\max_{1\leq j\leq n}\sum_{i=1}^{m}\left\vert a_{ij}\right\vert .
\]

The Hornich-Hlawka inequality can be recover from \ref{HHD} by considering the
particular case when $H=\mathbb{R}^{3}$ and%
\[
A=\left(
\begin{array}
[c]{ccc}%
0 & 1 & 1\\
1 & 0 & 1\\
1 & 1 & 0
\end{array}
\right)  .
\]
In this case $\bar{\alpha}_{m,n}(A)=1$ and $\left\Vert A\right\Vert _{1}=2.$

An unexpected generalization of the inequality (\ref{HHReal}) to the framework
of functions mixing suitable conditions of higher-order convexity has been
recently discovered by Ressel \cite{Res}. We state it here following
\cite{MN2023}:

\begin{theorem}
\label{thm_Ressel_alt}If $f:\mathbb{R}_{+}\rightarrow\mathbb{R}$ is a
continuous function of class $C^{3}$ on $(0,\infty)$ such that $f^{\prime}%
\geq0,$ $f^{\prime\prime}\leq0$ and $f^{\prime\prime\prime}\geq0,$ then%
\[
f\left(  \left\vert x\right\vert \right)  +f\left(  \left\vert y\right\vert
\right)  +f\left(  \left\vert z\right\vert \right)  +f\left(  \left\vert
x+y+z\right\vert \right)  \geq f\left(  \left\vert x+y\right\vert \right)
+f\left(  \left\vert y+z\right\vert \right)  +f\left(  \left\vert
z+x\right\vert \right)  +f(0)
\]
for all $x,y,z\in\mathbb{R}.$
\end{theorem}

Since all the functions $t^{\alpha}$ $($for $0<\alpha\leq1)$ fulfill the
hypotheses of Theorem \ref{thm_Ressel_alt}, it follows that
\[
\left\vert x\right\vert ^{\alpha}+\left\vert y\right\vert ^{\alpha}+\left\vert
z\right\vert ^{\alpha}+\left\vert x+y+z\right\vert ^{\alpha}\geq\left\vert
x+y\right\vert ^{\alpha}+\left\vert y+z\right\vert ^{\alpha}+\left\vert
z+x\right\vert ^{\alpha}\text{\quad}%
\]
for all $x,y,z\in\mathbb{R}$ and $\alpha\in(0,1].$ Combining this with
Corollary \ref{cor1} one obtains the following generalization of the
Hornich-Hlawka inequality:

\begin{theorem}
\label{thm_Ress_alpha}In every inner product space $E=(E,\left\Vert
\cdot\right\Vert _{2}),$%
\[
\left\Vert \mathbf{x}\right\Vert _{2}^{\alpha}+\left\Vert \mathbf{y}%
\right\Vert _{2}^{\alpha}+\left\Vert \mathbf{z}\right\Vert _{2}^{\alpha
}+\left\Vert \mathbf{x}+\mathbf{y}+\mathbf{z}\right\Vert _{2}^{\alpha}%
\geq\left\Vert \mathbf{x}+\mathbf{y}\right\Vert _{2}^{\alpha}+\left\Vert
\mathbf{y}+\mathbf{z}\right\Vert _{2}^{\alpha}+\left\Vert \mathbf{z}%
+\mathbf{x}\right\Vert _{2}^{\alpha}.
\]
for all $\mathbf{x},\mathbf{y},\mathbf{z}\in E$ and all $\alpha\in(0,1].$
\end{theorem}

Theorem \ref{thm_Ress_alpha} improves a previous result due to P. Ressel
\cite{Res}, Theorem 2, who noticed it, for the exponents $\alpha
\in\{2,1,1/2,1/2^{2},...\}.$

Applying Theorem \ref{thm_Ressel_alt} to the function $f(x)=\ln(1+x)$ one
obtains the multiplicative analog of the inequality (\ref{HHReal}),
\begin{multline*}
(1+\left\vert x\right\vert )(1+\left\vert y\right\vert )(1+\left\vert
z\right\vert )(1+\left\vert x+y+z\right\vert )\allowbreak\\
\geq(1+\left\vert x+y\right\vert )(1+\left\vert y+z\right\vert )(1+\left\vert
z+x\right\vert ),
\end{multline*}
for all $x,y,z\in\mathbb{R}.$ No extension of this remark to the case of inner
product spaces is known.

\begin{problem}
The Fréchet identity shows that \ the Theorem \ref{thm_Ress_alpha} also works
for $\alpha=2.$ For which other values {}{}of $\alpha\in(1,2)$ does this
theorem remain true?
\end{problem}

\section{The generalization of a result due to Ressel}

As was mentioned in the previous section, Ressel showed that the Fréchet
identity implies the truth of Theorem \ref{thm_Ress_alpha} for the exponents
$\alpha\in\{2,1,1/2,1/2^{2},...\}.$ A more general result was noticed by X.
Luo \cite{Luo2020}, who considered the framework of Banach spaces:

\begin{theorem}
\label{thm_1/2}Suppose $E$ is a vector space and $\varphi:E\rightarrow
\mathbb{R}$ is a function which verifies the condition%
\begin{equation}
\left\vert \varphi(\mathbf{x})-\varphi(\mathbf{y})\right\vert \leq
\varphi(\mathbf{x}+\mathbf{y})\leq\varphi(\mathbf{x})+\varphi(\mathbf{y})
\label{Luo}%
\end{equation}
for all $\mathbf{x},\mathbf{y}\in E$. If $\varphi^{2}$ verifies the
Hornich-Hlawka functional inequality, that is, if%
\begin{equation}
\sum\limits_{k=1}^{3}{\varphi}^{2}{\left(  \mathbf{x}_{k}\right)  }+{\varphi
}^{2}\left(  {\sum\limits_{k=1}^{3}}\mathbf{x}_{k}\right)  \geq\sum
\limits_{1\leq{i}<{j}\leq3}{\varphi}^{2}{\left(  \mathbf{x}{{_{i}+}}%
\mathbf{x}_{j}\right)  } \label{HH_functEq}%
\end{equation}
for all $\mathbf{x}_{1},\mathbf{x}_{2},\mathbf{x}_{3}\in E,$ then so does
$\varphi.$
\end{theorem}

Clearly, every function $\varphi$ which verifies the condition (\ref{Luo}) is
subadditive and nonnegative. Moreover, every even and subadditive function
$\varphi:E\rightarrow\mathbb{R}$ (in particular, every seminorm) verifies the
condition (\ref{Luo}).

Taking into account the Fréchet identity, it is now clear that Theorem
\ref{thm_1/2} implies the original Hornich-Hlawka inequality.

If $\varphi:E\rightarrow\mathbb{R}$ is even and subadditive and $S:\mathbb{R}%
_{+}\mathbb{\rightarrow R}_{+}$ is a concave function, then $S\circ\varphi$
verifies the condition (\ref{Luo}). Indeed, the function $S$ is subadditive
and increasing, which implies that $S\circ\varphi$ is an even and subadditive
function. This implies the following

\begin{corollary}
If $E=(E,\left\Vert \cdot\right\Vert )$ is a Hornich-Hlawka space, then
$\left\Vert \cdot\right\Vert ^{\alpha}$ verifies the Hornich-Hlawka inequality
for all exponents $\alpha\in\{1,1/2,1/2^{2},...\}.$
\end{corollary}

Our proof of Theorem \ref{thm_1/2} is based on the following auxiliary result,
inspired by a paper of D. M. Smiley and M. F. Smiley \cite{SmSm}.

\begin{lemma}
\label{lem_1/2}Let $a,b,c,d,x,y,z$ nonnegative real numbers such that the
following three conditions are fulfilled:

$(i)\ x\leq a+b,\ y\leq b+c~$and $z\leq c+a;$

$(ii)\ a+b+c\geq d~$and $a+b+c+d\geq\max\left\{  a+b+x,b+c+y,c+a+z\right\}  ;$

$(iii)\ a^{2}+b^{2}+c^{2}+d^{2}\geq x^{2}+y^{2}+z^{2}.$

Then $a+b+c+d\geq x+y+z.$
\end{lemma}

\begin{proof}
Indeed,%
\begin{multline*}
\left(  a+b+c\right)  ^{2}-d^{2}=a^{2}+b^{2}+c^{2}+2ab+2ac+2bc-d^{2}\\
\leq2\left(  a^{2}+b^{2}+c^{2}+ab+bc+ca\right)  -x^{2}-y^{2}-z^{2}\\
=\left(  a+b+x\right)  \left(  a+b-x\right)  +\left(  b+c+y\right)  \left(
b+c-y\right)  +\left(  c+a+z\right)  \left(  c+a-z\right) \\
\leq\left(  2a+2b+2c-x-y-z\right)  \max\left\{  a+b+x,b+c+y,c+a+z\right\}  ,
\end{multline*}
whence%
\begin{multline*}
\left(  a+b+c+d\right)  \left(  a+b+c-d\right) \\
\leq\left(  2a+2b+2c-x-y-z\right)  \max\left\{  a+b+x,b+c+y,c+a+z\right\}  .
\end{multline*}

Combining the last inequality with $\left(  ii\right)  $ we get%
\[
a+b+c-d\leq2a+2b+2c-x-y-z,
\]
that is,
\[
x+y+z\leq a+b+c+d.
\]

\end{proof}

\begin{proof}
[Proof of Thorem \ref{thm_1/2}]Put
\begin{align*}
a  &  =\varphi\left(  x_{1}\right)  ,\ b=\varphi\left(  x_{2}\right)
,\ c=\varphi\left(  x_{3}\right)  ,\ d=\varphi\left(  x_{1}+x_{2}%
+x_{3}\right)  ,\\
x  &  =\varphi\left(  x_{1}+x_{2}\right)  ,\ y=\varphi\left(  x_{2}%
+x_{3}\right)  ,\ z=\varphi\left(  x_{3}+x_{1}\right)  .
\end{align*}

Clearly $a,b,c,d,x,y,z$ are nonnegative numbers that verify condition $\left(
i\right)  $ of Lemma \ref{lem_1/2} because $\varphi$ is a subadditive
function. Condition $\left(  ii\right)  $ follows from the fact that $\varphi$
verifies the condition (\ref{Luo}), while condition $(iii)$ represents the
fact that $\varphi^{2}$ verifies the Hornich-Hlawka inequality.

Therefore the conclusion of Theorem \ref{thm_1/2} is a direct consequence of
Lemma \ref{lem_1/2}.
\end{proof}

\section{The case of arbitrary large families of elements}

The aforementioned paper of D. M. Smiley and M. F. Smiley \cite{SmSm} led\ P.
M. Vasi\'{c} and D. D. Adamovi\'{c} \cite{VA1968} to an inductive scheme
generating inequalities. We state here a slightly modified version of their
result as appeared in \cite{MPF}, Theorem 2, p. 528:

\begin{theorem}
\label{thmVA}Suppose that $\varphi$ is real-valued function defined on a
commutative additive semigroup $\mathcal{S}$ such that%
\[
\sum\limits_{1\leq{i}<{j}\leq3}{\varphi\left(  {{x{_{i}+x}}}_{j}\right)  }%
\leq\sum\limits_{k=1}^{3}{\varphi\left(  x_{k}\right)  }+{\varphi}\left(
{\sum\limits_{k=1}^{3}}x_{k}\right)
\]
for all $x_{1},x_{2},x_{3}\in\mathcal{S}.$ Then for each pair $\left\{
k,n\right\}  $ of integers with $2\leq k<n$ we also have%
\[
\sum\limits_{1\leq{i_{1}}<...<{i_{k}}\leq n}{\varphi\left(  {\sum
\limits_{j=1}^{k}{x{_{i_{j}}}}}\right)  }\leq\binom{n-2}{k-1}\sum
\limits_{k=1}^{n}{\varphi\left(  x_{k}\right)  }+\binom{n-2}{k-2}{\varphi
}\left(  {\sum\limits_{k=1}^{n}}x_{k}\right)
\]
whenever $x_{1},...,x_{n}\in\mathcal{S}.$
\end{theorem}

\begin{corollary}
Under the hypotheses of Theorem \emph{\ref{thmVA}},%
\[
\sum\limits_{k=2}^{n-1}\sum\limits_{1\leq{i_{1}}<...<{i_{k}}\leq n}%
{\varphi\left(  {\sum\limits_{j=1}^{k}{x{_{i_{j}}}}}\right)  }\leq\left(
2^{n-2}-1\right)  \left(  \sum\limits_{k=1}^{n}{\varphi\left(  x_{k}\right)
}+{\varphi}\left(  {\sum\limits_{k=1}^{n}}x_{k}\right)  \right)  .
\]

\end{corollary}

By combining Theorem \ref{thm_Ress_alpha} and Theorem \ref{thmVA} we infer the
following \emph{fractional} \emph{power inequalities }valid in the context of
real inner product spaces:%
\[
\sum_{1\leq i_{1}<\cdots<i_{k}\leq n}\left\Vert \mathbf{x}_{i_{1}}%
+\cdots+\mathbf{x}_{i_{k}}\right\Vert ^{\alpha}\leq\binom{n-2}{k-2}\left(
\frac{n-k}{k-1}\sum_{k=1}^{n}\left\Vert \mathbf{x}_{k}\right\Vert ^{\alpha
}+\left\Vert \sum_{k=1}^{n}\mathbf{x}_{k}\right\Vert ^{\alpha}\right)
\]
where $1\leq k\leq n-1,$ $n=3,4,5,...$ and $\alpha\in(0,1]$ $).$ This includes
(for $\alpha=1)$ the family of inequalities proved independently by D. Z.
Djokovi\'{c} \cite{Dj} and D. M. Smiley and M. F. Smiley \cite{SmSm}.

\section{Hornich-Hlawka operators?}

It is a common practice in mathematics to study the operator analogs of some
Banach space properties. The study of the Hornich Hlawka inequality leads
naturally to the following class of continuous and linear operators acting on
Banach spaces. A continuous linear operator $T:X\rightarrow Y$ is called a
\emph{Hornich-Hlawka operator} if there exists a constant $C>0$ such that%
\[
\left\Vert Tx+Ty\right\Vert +\left\Vert Ty+Tz\right\Vert +\left\Vert
Tz+Tx\right\Vert \leq\left\Vert T\right\Vert \left(  \left\Vert \mathbf{x}%
\right\Vert +\left\Vert \mathbf{y}\right\Vert +\left\Vert \mathbf{z}%
\right\Vert +\left\Vert \mathbf{x}+\mathbf{y}+\mathbf{z}\right\Vert \right)
\]
for all $x,y,z\in X.$

According to Theorem \ref{thmVA}, the Hornich-Hlawka operators also verify
inequalities of the form%
\[
\sum_{1\leq i_{1}<\cdots<i_{k}\leq n}\left\Vert Tx_{i_{1}}+\cdots+Tx_{i_{k}%
}\right\Vert \leq\binom{n-2}{k-2}\left\Vert T\right\Vert \left(  \frac
{n-k}{k-1}\sum_{k=1}^{n}\left\Vert x_{k}\right\Vert +\left\Vert \sum_{k=1}%
^{n}x_{k}\right\Vert \right)
\]
whenever $1\leq k\leq n-1$ and $n=3,4,5,...~.$

A Banach space $E$ is a Hornich-Hlawka space if the identity of $E$ is a
Hornich-Hlawka operator.

It is straightforward that every bounded linear operator that factors through
a Hornich-Hlawka space is a Hornich-Hlawka operator and the class of
Hornich-Hlawka operators forms an ideal of operators in the sense of Pietsch
\cite{Pietsch}.

It is too early to appreciate the usefulness of the Hornich-Hlavka type
operators, but the following problem deserves to be under attention:

\begin{problem}
Does every Hornich-Hlawka operator admit a factorization through a
Hornich-Hlawka space?
\end{problem}

\section{Further Comments}

We gain more insight into the Hornich-Hlawka inequality by appealing to the
harmonic analysis. The necessary information can be found in the book of C.
Berg, J. P. R. Christensen and P. Ressel \cite{BCR1984}, Chapter 4, Section 6.

For this, suppose that $C$ is a real linear space or a convex cone of it ($C$
may be open). A continuous function $f:C\rightarrow\mathbb{R}$ verifies the
\emph{Hornich-Hlawka functional equation} if%
\[
f\left(  \mathbf{x}\right)  +f\left(  \mathbf{y}\right)  +f\left(
\mathbf{z}\right)  +f\left(  \mathbf{x}+\mathbf{y}+\mathbf{z}\right)  \geq
f\left(  \mathbf{x}+\mathbf{y}\right)  +f\left(  \mathbf{y}+\mathbf{z}\right)
+f\left(  \mathbf{z}+\mathbf{x}\right)  +f(0)
\]
for all $\mathbf{x},\mathbf{y},\mathbf{z}\in C;$ The last term $f(0)$ can be
eliminated by replacing $f$ by $f-f(0)$ (or simply by deleting it if
$f(0)\geq0).$

Among the people who contributed to the study of this functional equation we
mention here T. Popoviciu \cite{Pop1946}, H. S. Witsenhausen \cite{W1973} and
W. Fechner \cite{Fe}.

When $C=(0,\infty)$, then both the completely monotone functions and the
Bernstein functions verifies \emph{Hornich-Hlawka functional equation}. See H.
S. Sendov and R. Zitikis \cite{SZ2014}, Corollary 4.3.

The significance of this functional equation is made clear by its connection
with iterated differences.

The \emph{difference operator} $\Delta_{\mathbf{h}}$ $($of step size
$\mathbf{h}\in C)$ associates to each function $f:E\rightarrow\mathbb{R}$ the
function $\Delta_{\mathbf{h}}f$ defined by the formula%
\[
\left(  \Delta_{\mathbf{h}}\mathbf{f}\right)  (\mathbf{x})=f(\mathbf{x}%
+\mathbf{h})-f(\mathbf{x}),\quad\text{for all }\mathbf{x},\mathbf{h}\in C.
\]
The difference operators are linear and commute to each other,%
\[
\Delta_{\mathbf{h}_{1}}\Delta_{\mathbf{h}_{2}}=\Delta_{\mathbf{h}_{2}}%
\Delta_{\mathbf{h}_{1}},\quad\text{for all }\mathbf{h}_{1},\mathbf{h}_{2}\in
C.
\]

A continuous function $f:C\rightarrow\mathbb{R}$ is called $\emph{n}%
$-\emph{monotone} if $f\geq0$ and%
\[
(-1)^{n}\Delta_{\mathbf{x}_{1}}\Delta_{\mathbf{x}_{2}}\cdots\Delta
_{\mathbf{x}_{n}}f(\mathbf{x})\geq0
\]
for all $\mathbf{x},\mathbf{x}_{1},...,\mathbf{x}_{n}\in C;$ the function
$f:C\rightarrow\mathbb{R}$ is called $n$-\emph{alternating} if it is
continuous and%

\[
(-1)^{n}\Delta_{\mathbf{x}_{1}}\Delta_{\mathbf{x}_{2}}\cdots\Delta
_{\mathbf{x}_{n}}f(\mathbf{x})\leq0
\]
for all $\mathbf{x},\mathbf{x}_{1},...,\mathbf{x}_{n}\in C.$ If $f$ is
$3$-alternating, then%
\begin{multline*}
\Delta_{\mathbf{x}}\Delta_{\mathbf{y}}\Delta_{\mathbf{z}}f(\mathbf{v}%
)=f\left(  \mathbf{x}+\mathbf{v}\right)  +f\left(  \mathbf{y}+\mathbf{v}%
\right)  +f\left(  \mathbf{z}+\mathbf{v}\right)  +f\left(  \mathbf{x}%
+\mathbf{y}+\mathbf{z}+\mathbf{v}\right) \\
-f\left(  \mathbf{x}+\mathbf{y}+\mathbf{v}\right)  -f\left(  \mathbf{y}%
+\mathbf{z}+\mathbf{v}\right)  -f\left(  \mathbf{z}+\mathbf{x}+\mathbf{v}%
\right)  -f(\mathbf{v})\geq0,
\end{multline*}
for all $\mathbf{x},\mathbf{y},\mathbf{z},\mathbf{v}\in C,$ which implies that
$f$ is a solution of the Hornich-Hlawka functional equation. In this way, the
study of this functional equation can be reduced to the study of
$3$-alternating function.

The paper of C. P. Niculescu and S. Sra \cite{NSra2023} includes several
example of functions of several variables that verify the Hornich-Hlawka
functional equation. In particular it is shown that the restriction of the
function $\det$ to the cone \textrm{Sym}$^{+}(n,\mathbb{R})$ of all $n\times
n$-dimensional positive semidefinite matrices has positive differences of any
order. This implies the determinantal Hornich-Hlawka inequality of M.
Lin~\cite{Lin},
\[
\det A+\det B+\det C+\det(A+B+C)\geq\det\left(  A+B\right)  +\det\left(
B+C\right)  +\det\left(  C+A\right)  ,
\]
for all $A,B,C\in\;$\textrm{Sym}$^{+}(n,\mathbb{R}).$

As is well known, the complete metric spaces with a global nonpositive
curvature display many nice properties that make them a natural generalization
of Hilbert spaces. See M. Ba\v{c}ák \cite{Bacak2014} and the references
therein. Except for a paper by D. Serre \cite{Serre}, nothing is known about
the Hornich-Hlawka inequalities in this context. No doubt, this qualifies as a
subject that deserves to be investigated.

\medskip

\noindent\textbf{\noindent Acknowledgement. }The authors would like to thank
\c{S}tefan Cobza\c{s} from University Babe\c{s}-Bolyai, Cluj-Napoca, for many
useful comments on the subject of this paper.

\end{document}